\documentclass{article}
\usepackage[a4paper, total={6in, 8in}]{geometry}
\usepackage{graphicx} 
\usepackage{amsmath,amssymb,xcolor,bm,amsthm,mathtools,comment, enumerate}

\usepackage{tikz}
\newtheorem{theorem}{Theorem}[section]

\newtheorem{remark}[theorem]{Remark}

\newtheorem{lemma}[theorem]{Lemma}

\newtheorem{conjecture}[theorem]{Conjecture}
\newtheorem{question}{Question}

\title{Matchings in the hypercube with specified edges}
\author{Joshua Erde\thanks{School of Mathematics, University of Birmingham, UK. Email: {\tt j.erde@bham.ac.uk}.}}

\begin{document}

\maketitle
\begin{abstract}
Given a matching $M$ in the hypercube $Q^n$, the \emph{profile} of $M$ is the vector $\bm{x}=(x_1,\ldots, x_n) \in \mathbb{N}^n$ such that $M$ contains $x_i$ edges whose endpoints differ in the $i$th coordinate. If $M$ is a perfect matching, then it is clear that $||\bm{x}||_1 = 2^{n-1}$ and it is easy to show that each $x_i$ must be even. Verifying a special case of a conjecture of Balister, Gy{\H{o}}ri, and Schelp, we show that these conditions are also sufficient.
\end{abstract}
\section{Introduction}
Balister, Gy{\H{o}}ri, and Schelp \cite{BGS11} considered the following partitioning problem: we are given a sequence $\bm{d_1},\ldots, \bm{d_{2^{n-1}}}$ of elements of $\mathbb{F}_2^n$ and our aim is to partition $\mathbb{F}_2^n$ into pairs $\{\bm{a_i},\bm{b_i}\}$ such that $\bm{a_i} - \bm{b_i} = \bm{d_i}$ for each $i \in \left[2^{n-1}\right]$. They noted that, since we are working over $\mathbb{F}_2^n$, for such a partition
\begin{equation}\label{e:null}
\sum_{i=1}^{2^{n-1}} \bm{d_i} = \sum_{i=1}^{2^{n-1}}  \bm{a_i} - \bm{b_i} =   \sum_{i=1}^{2^{n-1}}  \bm{a_i} + \bm{b_i} = \sum_{\bm{v} \in \mathbb{F}_2^n} \bm{v} = \bm{0}.
\end{equation}
They conjectured that this necessary condition was also sufficient.
\begin{conjecture}[Balister, Gy{\H{o}}ri, and Schelp]\label{c:BGS}
Let $n \geq 2$ and let $\bm{d_1},\ldots, \bm{d_{2^{n-1}}}$ be elements of $\mathbb{F}_2^n$ such that $\sum_{i=1}^{2^{n-1}} \bm{d_i} = \bm{0}$. Then there exists a partition of $\mathbb{F}_2^n$ into pairs $\{\bm{a_i},\bm{b_i}\}$ such that $\bm{a_i} - \bm{b_i} = \bm{d_i}$ for each $i \in \left[2^{n-1}\right]$.
\end{conjecture}
\noindent Similar problems of partitioning groups into pairs with prescribed differences have also been considered in $\mathbb{F}_p$ for $p>2$ prime \cite{KS10,PM09}, $\mathbb{F}_p^n$ \cite{KP12} and cyclic groups \cite{KS16}.

Balister, Gy{\H{o}}ri, and Schelp \cite{BGS11} showed that Conjecture \ref{c:BGS} holds under certain assumptions on the differences $\bm{d_i}$: namely if the first half of the differences are identical, and the second half come in identical pairs. More recently, Kov{\'a}cs \cite{K23,K23a} gave further evidence towards Conjecture \ref{c:BGS}.
\begin{theorem}[Kov{\'a}cs]\label{t:kovacs}
If either
\begin{enumerate}[(a)]
    \item\label{i:one} the number of distinct $\bm{d}_i$ is at most $n-2\log_2 n -1$; or
    \item\label{i:two} $n$ is sufficiently large and at least a $\frac{28}{29}$ proportion of the $\bm{d}_i$ are identical,
\end{enumerate}
then Conjecture \ref{c:BGS} holds
\end{theorem}

In this note we consider another natural restriction of Conjecture \ref{c:BGS}, where we insist that each difference $\bm{d_i}$ must be some unit vector, i.e., of the form $\bm{e_i}$ for some $i$. We note that, since there are precisely $n$ unit vectors, this is not quite covered by Theorem \ref{t:kovacs} \eqref{i:one}. In this case, Conjecture \ref{c:BGS} has a simple reformulation in terms of matchings in the hypercube $Q^n$.

Given a matching $M \subseteq Q^n$ we say the \emph{profile} of $M$ is the vector $\bm{x}=(x_1,\ldots, x_n) \in \mathbb{N}^n$ such that $M$ contains $x_i$ edges whose endpoints differ in the $i$th coordinate, and we are interested in which tuples $\bm{x} \in \mathbb{N}^n$ are achievable as the profile of some matching. We call such tuples \emph{admissible}. In this case, the obvious necessary condition for a tuple to be the profile of a perfect matching is that each unit vector appears an even number of times as a matching edge, i.e., each $x_i$ is even, in which case we say that $\bm{x}$ is \emph{even}, and Conjecture \ref{c:BGS} would imply that this condition is also sufficient. The aim of this short note is to verify the conjecture in this case. Let us call a tuple $\bm{x} \in \mathbb{N}^n$  \emph{perfect} if $||\bm{x}||_1 := \sum_i x_i = 2^{n-1}$.

\begin{theorem}\label{t:perfect}
Let $n \geq 2$. A perfect tuple $\bm{x} \in \mathbb{N}^n$ is admissible if and only if it is even.
\end{theorem}

After distributing an initial version of the manuscript it was brought to the author's attention, that Conjecture \ref{c:BGS} and related questions have also been considered independently in other contexts. Firstly, in the setting of \emph{set-sequential trees} where work of Golowich and Kim \cite{GK20} shows that Conjecture \ref{c:BGS} holds when there are at most $n$ distinct $\bm{d}_i$, which implies Theorem \ref{t:perfect}. Similar questions have also been considered in the context of \emph{batch codes} in information theory, although a subtle difference here is that a solution to a request vector does not correspond to a matching in $Q^d$, but rather a subgraph in which every vertex \emph{except} the origin has degree at most one. In this setting, Wang, Kiah and Cassuto \cite{WKC15} and Wang, Kiah, Cassuto and Bruckner \cite{WKCB17} prove a statement analogous, but not equivalent, to Theorem \ref{t:perfect} using a very similar inductive argument. More recently Hollmann, Khathuria, Riet, and Skachek \cite{HKAES23} observed that an old theorem of Hall \cite{HJR52} can be used to prove that certain codes are batch codes, and in particular a straightforward adaptation of their methods shows that 
Conjecture \ref{c:BGS} holds in the case where all the vectors $\bm{d_i}$ lie outside of some fixed hyperplane $H$ (see \cite[Remark 5.4]{HKAES23}), which also implies Theorem \ref{t:perfect}. However, since we believe the reformulation in terms of matchings in $Q^n$ leads to interesting questions (see Section \ref{s:discussion}), and our proof is short and simple, we hope the work in this note may still be of interest.

The structure of the paper is as follows. In Section \ref{s:main} we prove Theorem \ref{t:perfect} and in Section \ref{s:discussion} we discuss some related questions.

\section{Main result}\label{s:main}
We first note that by symmetry the ordering of the coordinates of $\bm{x}$ is irrelevant, and so we may assume from this point forwards that the coordinates appear in non-decreasing order, i.e., $x_1 \leq x_2 \leq \ldots \leq x_n$. Given $\bm{x},\bm{z} \in \mathbb{N}^{n}$ we say that \emph{$\bm{z}$ precedes $\bm{x}$}, which we write as $\bm{z} \preceq \bm{x}$, if $z_i \leq x_i$ for all $i \in [n]$. Clearly the set of admissible tuples is down-closed in the partial order $\preceq$. In particular, Theorem \ref{t:perfect} is equivalent to the statement that every even tuple is admissible. Finally, given $\bm{x} \in \mathbb{N}^{n}$ and $\bm{z} \in \mathbb{N}^{m}$ we will write $(\bm{x},\bm{z})$ for the vector in $\mathbb{N}^{n+m}$ obtained by concatenation $\bm{x}$ and $\bm{z}$.

The following simple observation underpins our proof.
\begin{lemma}\label{l:admissible}
If $\bm{x} \in \mathbb{N}^{n}$ is admissible and $x_{n+1} = 2^{n} - 2||\bm{x}||_1$, then $(2\bm{x},x_{n+1})$ is a perfect admissible tuple.
\end{lemma}
\begin{proof}
By assumption, there is a matching $M_1 \subseteq Q^{n}$ with profile $\bm{x}$. We form a matching $M_2 \subseteq Q^{n+1}$ by taking disjoint two copies of $M_1$, that is,
\[
M_2 := \{\big((\bm{u},0),(\bm{v},0)\big), \big((\bm{u},1),(\bm{v},1)\big) \colon (\bm{u},\bm{v}) \in M_1 \}
\]
and note that $M_2$ has profile $
(2\bm{x},0)$.

Let $X \subseteq V\left(Q^{n}\right)$ be the set of vertices which are not covered by $M_1$, where we note that $|X| = 2^{n} - 2||\bm{x}||_1 = x_{n+1}$. We form $M_3$ by adding to $M_2$ the edges in direction $\bm{e_{n+1}}$ joining the two copies of each vertex in $X$ in $Q^{n+1}$. That is, 
\[
M_3 := M_2 \cup \{ \big((\bm{v},0), (\bm{v},1) \big) \colon \bm{v} \in X \}.
\]

It is clear that $M_3 \subseteq Q^n$ is a matching with profile $(2\bm{x}, x_{n+1})$ and since $2||\bm{x}||_1 + |X| = 2^{n}$ it follows that $(2\bm{x}, x_{n+1})$ is a  perfect admissible tuple.
\end{proof}

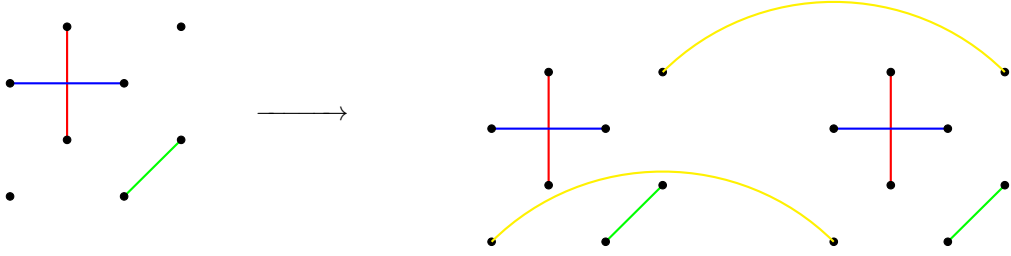
\begin{figure}[!ht]
\centering
\begin{minipage}{.2\textwidth}
\begin{tikzpicture}[scale=1.5]

\draw[red,thick] (0.5,0.5)--(0.5,1.5);
\draw[blue,thick] (0,1)--(1,1);
\draw[green,thick] (1,0)--(1.5,0.5);

 \node[shape=circle,draw=black, fill, scale=0.3] at (0,0){};
  \node[shape=circle,draw=black, fill, scale=0.3] at (1,0){};
   \node[shape=circle,draw=black, fill, scale=0.3] at (0,1){};
    \node[shape=circle,draw=black, fill, scale=0.3] at (1,1) {};
     \node[shape=circle,draw=black, fill, scale=0.3] at (0.5,0.5) {};
      \node[shape=circle,draw=black, fill, scale=0.3] at (1.5,0.5) {};
       \node[shape=circle,draw=black, fill, scale=0.3] at (0.5,1.5) {};
        \node[shape=circle,draw=black, fill, scale=0.3] at (1.5,1.5) {};
\end{tikzpicture}
\end{minipage}
\begin{minipage}{.2\textwidth}
\begin{tikzpicture}
\node at (0,0.5) {$\xrightarrow{\hspace*{1cm}}$};
\end{tikzpicture}
\end{minipage}
\begin{minipage}{.4\textwidth}
\begin{tikzpicture}[scale=1.5]

\draw[red,thick] (0.5,0.5)--(0.5,1.5);
\draw[blue,thick] (0,1)--(1,1);
\draw[green,thick] (1,0)--(1.5,0.5);

 \node[shape=circle,draw=black, fill, scale=0.3] at (0,0){};
  \node[shape=circle,draw=black, fill, scale=0.3] at (1,0){};
   \node[shape=circle,draw=black, fill, scale=0.3] at (0,1){};
    \node[shape=circle,draw=black, fill, scale=0.3] at (1,1) {};
     \node[shape=circle,draw=black, fill, scale=0.3] at (0.5,0.5) {};
      \node[shape=circle,draw=black, fill, scale=0.3] at (1.5,0.5) {};
       \node[shape=circle,draw=black, fill, scale=0.3] at (0.5,1.5) {};
        \node[shape=circle,draw=black, fill, scale=0.3] at (1.5,1.5) {};
        
\draw[red,thick] (3.5,0.5)--(3.5,1.5);
\draw[blue,thick] (3,1)--(4,1);
\draw[green,thick] (4,0)--(4.5,0.5);

 \node[shape=circle,draw=black, fill, scale=0.3] at (3,0){};
  \node[shape=circle,draw=black, fill, scale=0.3] at (4,0){};
   \node[shape=circle,draw=black, fill, scale=0.3] at (3,1){};
    \node[shape=circle,draw=black, fill, scale=0.3] at (4,1) {};
     \node[shape=circle,draw=black, fill, scale=0.3] at (3.5,0.5) {};
      \node[shape=circle,draw=black, fill, scale=0.3] at (4.5,0.5) {};
       \node[shape=circle,draw=black, fill, scale=0.3] at (3.5,1.5) {};
        \node[shape=circle,draw=black, fill, scale=0.3] at (4.5,1.5) {};

        \draw[yellow,thick] (0,0) to[out=45,in=135] (3,0) (1.5,1.5) to[out=45,in=135] (4.5,1.5);
\end{tikzpicture}
\end{minipage}
\hfill
\caption{Example showing how to build matching with profile $(2,2,2,2)$ from a matching with profile $(1,1,1)$}\label{f:3dim}
\end{figure}

Lemma \ref{l:admissible} lends itself well to an inductive argument\,---\,given a perfect even tuple $\bm{x} \in \mathbb{N}^{n+1}$, if $x_i \equiv 0 \mod 4$ for all $i \in [n]$, then $\bm{y} := \left( \frac{x_1}{2},\frac{x_2}{2},\ldots,\frac{x_n}{2}\right)$ is an even tuple in $\mathbb{N}^{n}$ and we can deduce the admissability of $\bm{x}$ from that of $\bm{y}$. However, it is not clear what to do if some of the coordinates of $\bm{y}$ are odd.

The key idea here is that, if $x_{n+1}$ is sufficiently large, then $\bm{y}$ will not be perfect, and so we might hope to find an even tuple $\bm{y'}$ which $\preceq$-dominates $\bm{y}$ with $||\bm{y'}||_1 \leq 2^{n-1}$. The admissability of $\bm{y'}$ would then imply the admissability of $\bm{y}$ and thus the admissability of $\bm{x}$.

So, for example, to show that $(0,2,4,4,6)$ is admissible, by Lemma \ref{l:first} it is sufficient to show that $(0,1,2,2)$ is admissible, however $(0,1,2,2) \preceq (2,2,2,2)$ and we know (cf.~Figure \ref{f:3dim}) that $(2,2,2,2)$ is admissible.

Let us make the preceding idea formal with the following lemma. Given $\bm{x} \in \mathbb{N}^n$, let us write $o(\bm{x})$ for the number of coordinates of $\bm{x}$ except the last which are congruent to $2 \mod 4$.

\begin{lemma}\label{l:first}
Let $n\geq 2$, and let $\bm{x} \in \mathbb{N}^{n+1}$ be even with $||\bm{x}||_1 \leq 2^{n}$ and $x_{n+1} \geq 2 o(\bm{x})$. If $\bm{y} := \left(\frac{x_1}{2},\frac{x_2}{2}, \ldots, \frac{x_n}{2}\right)$, then there exists a perfect even $\bm{y'} \in \mathbb{N}^n$ such that $\bm{y} \preceq \bm{y'}$
\end{lemma}
\begin{proof}
For each $i \in [n]$ let 
\[
o_i = \begin{cases} 1 \qquad &\text{ if } x_i \equiv 2 \mod 4,\\
0 \qquad &\text{ otherwise}.\end{cases}
\]
Let $\bm{z} = \bm{y} + \bm{o}$. Note, in particular, that $\bm{y} \preceq \bm{z}$ and that by construction $\bm{z}$ is even.

Since $||\bm{o}||_1 = o(\bm{x}) \leq \frac{x_{n+1}}{2}$, it follows that 
\[
||\bm{z}|| \leq \sum_{i=1}^n \frac{x_i}{2} + o_i \leq \frac{||\bm{x}||_1}{2} \leq 2^{n-1}.
\]
Hence, since $2^{n-1}$ is even, there is some perfect even $\bm{y'} \succeq \bm{z} \succeq \bm{y}$ as claimed.
\end{proof}

\begin{proof}[Proof of Theorem \ref{t:perfect}]
We first note that by \eqref{e:null} every perfect admissible tuple is even.

For the other direction, let us suppose towards a contradiction that the theorem does not hold, and let $\bm{x} \in \mathbb{N}^{n+1}$ be a perfect even tuple which is not admissible, with $n$ as small as possible. Let $\bm{y} = \left(\frac{x_1}{2},\frac{x_2}{2},\ldots, \frac{x_n}{2}\right)$. Note that, $\bm{x} = (2\bm{y},x_{n+1})$ and, since $\bm{x}$ is perfect, $x_{n+1} = 2^n - 2||\bm{y}||_1$. Furthremore, since $\bm{x}$ is not admissible, by Lemma \ref{l:admissible} neither is $\bm{y}$.

If $\bm{x}$ is such that $x_{n+1} \geq 2n$, then since $o(\bm{x}) \leq n$, it follows from Lemma \ref{l:first} that there exists a perfect even $\bm{y'} \succeq \bm{y}$. However, since $\bm{y}$ is not admissible, neither is $\bm{y'} \in \mathbb{N}^n$, contradicting our choice of $\bm{x}$.

Hence, we may assume that $x_{n+1} < 2n$. However, if $n+1 \geq 8$, then since $x_1\leq x_2 \leq \ldots \leq x_{n+1}$,
\begin{equation} \label{e:dim}
x_{n+1} \geq \frac{2^n}{n+1} \geq 2n,
\end{equation}
and hence we may assume that $n+1 \leq 7$.

This reduces the proof to a finite case check, and in fact more precise applications of Lemma \ref{l:first} will deal with all but three of the remaining cases. In the table below we list all perfect even tuples $\bm{x} \in \mathbb{N}^{n+1}$ with $3\leq n+1 \leq 7$ such that $x_{n+1} < 2n$. For tuples listed in blue, $x_{n+1} \geq 2o(\bm{x})$, and so by Lemma \ref{l:first} these tuples cannot be minimal counterexamples. The remaining tuples are listed in red.

\begin{center}
\begin{tabular}{ |c|c| } 
 \hline
 $n+1=3$ & \color{blue}$(0,2,2)$\color{black}  \\ 
 $n+1=4$ &  \color{blue}$(0,0,4,4)$\color{black}, \color{blue}$(0,2,2,4)$\color{black}, \color{red}$(2,2,2,2)$  \\ 
$n+1=5$ &  \color{blue}$(0,0,4,6,6)$\color{black}, \color{blue}$(0,2,2,6,6)$\color{black}, \color{blue}$(0,2,4,4,6)$\color{black},\\
& \color{blue}$(2,2,2,4,6)$\color{black}, \color{blue}$(0,4,4,4,4)$\color{black}, \color{blue}$(2,2,4,4,4)$ \color{black} \\ 
$n+1=6$ &  \color{blue}$(0,0,8,8,8,8)$\color{black},\color{blue}$(0,2,6,8,8,8)$\color{black},\color{blue}$(0,4,4,8,8,8)$\color{black},\color{blue}$(2,2,4,8,8,8)$\color{black}, \\
&\color{blue}$(0,4,6,6,8,8)$\color{black}, \color{blue}$(2,2,6,6,8,8)$\color{black},\color{blue}$(2,4,4,6,8,8)$\color{black},\\
& \color{blue}$(4,4,4,4,8,8)$\color{black},
\color{blue}$(0,6,6,6,6,8)$\color{black}, \color{blue}$(2,4,6,6,6,8)$, \\
& \color{blue}$(4,4,4,6,6,8)$\color{black},
\color{red}$(2,6,6,6,6,6)$\color{black}, \color{blue}$(4,4,6,6,6,6)$\color{black}\\ 
 $n+1=7$ &  \color{blue}$(4,10,10,10,10,10,10)\color{black}$,\color{blue}$(6,8,10,10,10,10,10)$\color{black},\color{blue}$(8,8,8,10,10,10,10)$\color{black}  \\ 
 \hline
\end{tabular}
\end{center}
So, the only three cases we have to construct by hand are $(0,2)$, $(2,2,2,2)$ and $(2,6,6,6,6,6)$. At this point it is slightly easier to apply Lemma \ref{l:admissible} to see that it remains to show that $(1)$, $(1,1,1)$ and $(3,3,3,3,1)$ are admissible, where in fact we show the stronger statement that $(3,3,3,3,3)$ is admissible.

The first is trivial. The second is only slightly less trivial and is given already in Figure \ref{f:3dim}. The third is a fun exercise and one possible solution is given in Figure \ref{f:5dim}.

\begin{figure}[!ht]
\centering
\begin{tikzpicture}

\draw[red,thick] (0,4)--(1,4) (4.5,1.5)--(5.5,1.5) (4,1)--(5,1);
\draw[blue,thick] (1,0)--(1,1) (4.5,4.5)--(4.5,5.5) (0.5,4.5)--(0.5,5.5);
\draw[green,thick] (0,1)--(0.5,1.5) (5,4)--(5.5,4.5) (5,0)--(5.5,0.5);
\draw[yellow,thick] (0.5,0.5) to[out=45,in=135] (4.5,0.5) (1,5) to[out=45,in=135] (5,5) (0,5) to[out=45,in=135] (4,5);
\draw[magenta,thick] (4,0) to[out=135,in=225] (4,4) (1.5,1.5) to[out=135,in=225] (1.5,5.5) (1.5,0.5) to[out=135,in=225] (1.5,4.5);

\node[shape=circle,draw=black, fill, scale=0.3] at (0,0){};
  \node[shape=circle,draw=black, fill, scale=0.3] at (1,0){};
   \node[shape=circle,draw=black, fill, scale=0.3] at (0,1){};
    \node[shape=circle,draw=black, fill, scale=0.3] at (1,1) {};
     \node[shape=circle,draw=black, fill, scale=0.3] at (0.5,0.5) {};
      \node[shape=circle,draw=black, fill, scale=0.3] at (1.5,0.5) {};
       \node[shape=circle,draw=black, fill, scale=0.3] at (0.5,1.5) {};
        \node[shape=circle,draw=black, fill, scale=0.3] at (1.5,1.5) {};

        \node[shape=circle,draw=black, fill, scale=0.3] at (4,0){};
  \node[shape=circle,draw=black, fill, scale=0.3] at (5,0){};
   \node[shape=circle,draw=black, fill, scale=0.3] at (4,1){};
    \node[shape=circle,draw=black, fill, scale=0.3] at (5,1) {};
     \node[shape=circle,draw=black, fill, scale=0.3] at (4.5,0.5) {};
      \node[shape=circle,draw=black, fill, scale=0.3] at (5.5,0.5) {};
       \node[shape=circle,draw=black, fill, scale=0.3] at (4.5,1.5) {};
        \node[shape=circle,draw=black, fill, scale=0.3] at (5.5,1.5) {};

        \node[shape=circle,draw=black, fill, scale=0.3] at (0,4){};
  \node[shape=circle,draw=black, fill, scale=0.3] at (1,4){};
   \node[shape=circle,draw=black, fill, scale=0.3] at (0,5){};
    \node[shape=circle,draw=black, fill, scale=0.3] at (1,5) {};
     \node[shape=circle,draw=black, fill, scale=0.3] at (0.5,4.5) {};
      \node[shape=circle,draw=black, fill, scale=0.3] at (1.5,4.5) {};
       \node[shape=circle,draw=black, fill, scale=0.3] at (0.5,5.5) {};
        \node[shape=circle,draw=black, fill, scale=0.3] at (1.5,5.5) {};

        \node[shape=circle,draw=black, fill, scale=0.3] at (4,4){};
  \node[shape=circle,draw=black, fill, scale=0.3] at (5,4){};
   \node[shape=circle,draw=black, fill, scale=0.3] at (4,5){};
    \node[shape=circle,draw=black, fill, scale=0.3] at (5,5) {};
     \node[shape=circle,draw=black, fill, scale=0.3] at (4.5,4.5) {};
      \node[shape=circle,draw=black, fill, scale=0.3] at (5.5,4.5) {};
       \node[shape=circle,draw=black, fill, scale=0.3] at (4.5,5.5) {};
        \node[shape=circle,draw=black, fill, scale=0.3] at (5.5,5.5) {};

\end{tikzpicture}
\caption{A matching in $Q^5$ with profile $(3,3,3,3,3)$.}\label{f:5dim}
\end{figure}
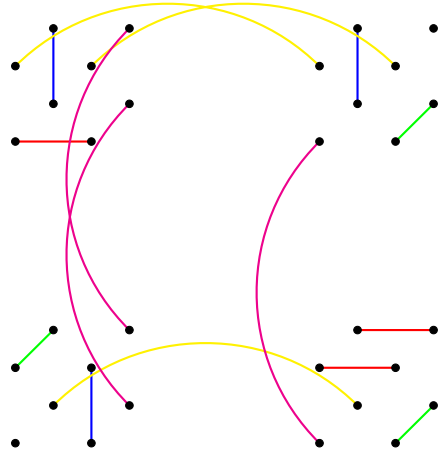
\end{proof}

\section{Discussion}\label{s:discussion}
As we noted in Section \ref{s:main}, if $\bm{x}$ is an admissible tuple and $\bm{y} \preceq \bm{x}$, then $\bm{y}$ is also admissible. In particular, Theorem \ref{t:perfect} also tells us a lot about the admissible profiles of non-perfect matchings. 

However, whilst Theorem \ref{t:perfect} fully characterises the \emph{perfect} admissible tuples, there are admissible tuples which are not $\preceq$-dominated by any perfect admissible tuples, for example two of the base cases in the proof of Theorem \ref{t:perfect}. It is relatively easy to see that Conjecture \ref{c:BGS} would imply that \emph{all} tuples $\bm{x} \in \mathbb{N}^n$ with $||\bm{x}||_1 < 2^{n-1}$ are admissible.

In fact, it was pointed out by a referee that this already follows from the work of Hollmann, Khathuria, Riet, and Skachek \cite{HKAES23}.

\begin{theorem}\label{c:admissible}
Let $n \in \mathbb{N}$. Every tuple $\bm{x} \in \mathbb{N}^n$ with $||\bm{x}||_1 < 2^{n-1}$ is admissible.
\end{theorem}
\begin{proof}
We first note that, by similar arguments as in \cite[Remark 5.4]{HKAES23}, a result of Hall \cite{HJR52} implies that Conjecture \ref{c:BGS} holds if all the vectors $\bm{d}_1,\ldots,\bm{d}^{2^{n-1}}$ have odd Hamming weight. 

Clearly it suffices to prove the statement for $\bm{x} \in \mathbb{N}^n$ with $||\bm{x}||_1 = 2^{n-1}-1$, and given any such vector $\bm{x}$ we can consider the sequence $\bm{d}_1,\ldots,\bm{d}^{2^{n-1}}$ given by $x_j$ copies of $\bm{e}_j$ for each $j \in [n]$ together with $\bm{d}_{2^{n-1}} = \sum_{i=1}^{2^{n-1}-1} \bm{d}_i$. Since each $\bm{d}_i$ with $i \in \left[2^{n-1}-1\right]$ has Hamming weight one, it follows that each $\bm{d}_i$ with $i\in \left[2^{n-1}\right]$ has odd Hamming weight, and hence there is a partition of $\mathbb{F}^n_2$ into pairs $\{\bm{a}_i,\bm{b}_i\}$ such that $\bm{a}_i - \bm{b}_i = \bm{d}_i$ for all $i \in \left[2^{n-1}\right]$. The matching $
M = \left\{ (\bm{a}_i,\bm{b}_i) \colon  \in \left[2^{n-1}-1\right]\right\}$ then shows that $\bm{x}$ is admissible.
\end{proof}


It would also be interesting to characterise which tuples are the profiles of Hamilton cycles in $Q^d$.
Since each Hamilton cycle can be split into two matchings, it is clear that any such tuple must be the sum of two perfect admissible tuples, and so in particular must also be even. Furthermore, clearly no direction can appear more than $2^{n-1}$ times, Finally, if some $x_i=0$, then all edges are either contained in the hyperplane $H_i =\{ \bm{a} \in \{0,1\}^n \colon a_i=0\}$ or its complement, which is a contradiction as a Hamilton cycle is connected and spanning. Hence, each $x_i \geq 1$ and so, since $\bm{x}$ must be even, each $x_i \geq 2$. We conjecture that these are the only restrictions.
\begin{conjecture}\label{c:hamiltonian}
Let $n \in \mathbb{N}$. Every even tuple $\bm{x} \in \mathbb{N}^n$ with $||\bm{x}||_1 = 2^n$, $\max_i \{x_i\} \leq 2^{n-1}$ and $\min_i \{x_i\} \geq 2$ is the profile of a Hamilton cycle.
\end{conjecture}

\begin{remark}
The author has subsequently been made aware of the work in \cite{P21,P12}. In particular, it is necessary that for every $i_1,\ldots, i_k \in [n]$, $\sum_{j=1}^k x_{i_j} \geq 2^k$. Indeed, the projection of a Hamilton cycle to the subcube spanned by the coordinates $\{i_1,\ldots,i_k\}$ gives a closed walk visiting each vertex and hence require at least $2^k$ edges. Note that the two conditions $\max_i \{x_i\} \leq 2^{n-1}$ and $\min_i \{x_i\} \geq 2$ in Conjecture \ref{c:hamiltonian} are implied by this condition.

On the other hand, it is shown in \cite{P21,P12} that for sufficiently large $n$, this condition together with the fact that  $||\bm{x}||_1 = 2^n$ and that $\bm{x}$ is even is sufficient to guarantee the existence of a Hamilton cycle with profile $\bm{x}$.
\end{remark}

A different generalisation would be to consider decomposing the hypercube into $2$-dimensional faces with specified orientations. More concretely, given a decomposition $D$ of $E\left(Q^n\right)$ into $4$-cycles, the \emph{profile} of this decomposition is the weighting $w$ of $K_n$ where the weight $w(ij)$ of an edge is the number of $4$-cycles in $D$ where the antipodal points differ in coordinates $i$ and $j$. It is easy to see that \eqref{e:null} implies that for any $i$ the total number of faces with an edge in direction $i$ must be even, i.e., $\sum_{j \neq i} w(ij) = 0 \mod 2$. Indeed, if we consider the intersection of such a decomposition with the subcube whose $i$th coordinate is $0$ then we can easily construct a perfect matching with $\sum_{j \neq i} w(ij)$ edges in direction $i$. However, it is a simple exercise to check that the weighting $w$ of $K_4$ given by \[
w(ij) = \begin{cases}
    2 \qquad \text{ if } ij= 12 \text{ or } 34,\\
    0 \qquad \text{ otherwise}
    \end{cases}
    \]
is not the profile of any such decomposition, and so, at least in small dimensions, there are other obstructions.
\begin{question}
What are the admissible profiles for decompositions of $Q^n$ into $4$-cycles? 
\end{question}

Another interesting question would be to characterise the admissible profiles for matchings of the \emph{middle layer graph}, the induced subgraph $M_n$ of $Q^{2n+1}$ on the middle two layers. A similar argument as in \eqref{e:null} shows that if $\bm{x}$ is the profile of a perfect matching on $M_n$ then $x_i \equiv \binom{2n+1}{n+1} \mod 2$ for each $i\in [2n+1]$. An application of Lucas' theorem shows that these binomial coefficients are odd if and only if $n+1$ is a power of $2$.

However, there is another necessary condition beyond this parity condition and the trivial condition that $||\bm{x}||_1 = \frac{|V(M_n)|}{2}$. Indeed, since only $\binom{2n}{n}$ vertices in each layer are incident to an edge in a fixed direction,  we require that $x_i \leq \binom{2n}{n}$ for each $i \in [2n+1]$. In fact, since the $\binom{2n}{n-1}$ vertices in the lower layer which are not adjacent to an edge in direction $\bm{e_i}$ are only adjacent to the vertices in the upper layer which are adjacent to an edge in direction $\bm{e_i}$, and therefore must be matched to a subset of these vertices via edges in some direction not equal to $\bm{e_i}$, the matchings can contain at most $\binom{2n}{n} - \binom{2n}{n-1} = \binom{2n-1}{n-1}$ edges in direction $\bm{e_i}$.

For $n=1$, it is easy to verify that these two conditions, together with the trivial condition that $||\bm{x}||_1 = \frac{|V(M_n)|}{2} = \binom{2n+1}{n+1}$, are also sufficient, since $(1,1,1)$ is the only tuple satisfying these conditions.

\begin{question}
What are the admissible profiles of perfect matchings of $M_n$? 
\end{question}

Finally, another intriguing generalisation would be to look at matchings on other polytopes. For example, each of the edges in the $n$-dimensional permutahedron Perm$(n)$, embedded in $\mathbb{R}^{n+1}$ as the convex hull of all \emph{permutation vectors}, is parallel to some vector of the form $\bm{e_i} - \bm{e_j}$ with $i,j \in [n+1]$. Let us define the \emph{profile} of a matching of Perm$(n)$ to be the weighting $w$ of $K_{n+1}$ where $w(ij)$ is the number of edges parallel to the vector $\bm{e_i} - \bm{e_j}$.
\begin{question}
What are the admissible profiles of perfect matchings of Perm$(n)$?
\end{question}
Note that, since Perm$(n)$ can also be viewed as the Cayley graph of $S_{n+1}$ whose generating set is the set of adjacent transposition $S = \{ (i,i+1) \colon i \in [n]\}$, it is possible to phrase this as a question similar to Conjecture \ref{c:BGS} about the symmetric group $S_{n+1}$, where the `differences' between permutations are required to be adjacent transpositions.

In the case $n=2$ there are only $2$ matchings, each of which use an edge in each direction once. For $n=3,4$ computational evidence suggests that each direction must be used an even number of times, and that, if we view the profile $w$ as a vector in $\mathbb{R}^{\binom{n+1}{2}}$, then the set of admissible vectors are the even interior points of some convex polytope. For $n=3$, this polytope is determined by the inequalities
\begin{align*}
&w(12) + w(34) = w(13) + w(24) = w(14)+w(23)=4;\\
&w(ij) \geq 0 \text{ for all } i,j,
\end{align*}
which essentially say that each pair of orthogonal directions has to be used exactly four times in the matching. For $n=4$ we were able to calculate a list of $1080$ inequalities which determine the polytope, but it is not clear if there is a more enlightening description.

\subsection*{Acknowledgements}
The author would like to thank Maur\'{i}cio Collares for pointing out the relevance of Lucas' theorem, as well as for providing computational evidence to motivate the questions in Section \ref{s:discussion} and Thang Nguyen for bringing the references \cite{P21,P12} to his attention. The author would also like to thank the anonymous referees for their valuable suggestions and in particular for providing a proof of Theorem \ref{c:admissible}. The author was supported in part by the Austrian Science Fund (FWF) [10.55776/\text{P36131}].

\bibliographystyle{plain}
\bibliography{matching} 
\end{document}